\documentclass[referee,envcountsect,]{svjour3}
\smartqed
\usepackage{graphicx}
\usepackage{xcolor}

\usepackage{amsmath,amsfonts,amssymb}
\usepackage{tikz}
\usepackage{mdframed}
\newmdenv[linecolor=black,skipabove=\topsep,skipbelow=\topsep,
leftmargin=-5pt,rightmargin=-5pt,
innerleftmargin=5pt,innerrightmargin=5pt]{mybox}

\begin{document}
	
	\title{Discontinuous Strongly Quasiconvex Functions}
	
	
	\author{Nguyen Thi Van Hang\and Felipe Lara\and \\Nguyen Dong Yen}
	
	\institute{Nguyen Thi Van Hang \at
		Institute of Mathematics, Vietnam Academy of Science and Technology, 18 Hoang Quoc Viet, Hanoi 10072, Vietnam\\ ntvhang@math.ac.vn
		\and
		Felipe Lara \at
		Instituto de Alta investigación (IAI), Universidad de Tarapacá, Arica, Chile\\felipelaraobreque@gmail.com; flarao@academicos.uta.cl\\ ORCID-ID: 0000-
		0002-9965-092
		\and
		Nguyen Dong Yen (Corresponding author)\at Institute of Mathematics, Vietnam Academy of Science and Technology, 18 Hoang Quoc Viet,
		Hanoi 10072, Vietnam\\
		ndyen@math.ac.vn}
	
	\date{Received: date / Accepted: date}
	
	\maketitle
	
	\begin{abstract} A fundamental open question asking whether all real-valued strongly quasiconvex functions defined on $\mathbb R^n$ are necessarily continuous, akin to their convex counterparts, is answered in detail in this paper. Among other things, we show that such functions can have infinitely many points of discontinuity. The failure of lower semicontinuity together with the lack of upper semicontinuity at infinitely many points of certain real-valued strongly quasiconvex functions are also shown. 
	\end{abstract}
	
	\medskip
	

 \section{Introduction}\label{Sect_1}
 
 An extended-real-valued function defined on a topological vector space is said to be \textit{quasiconvex} if its sublevel sets are convex. Being of great theoretical importance, quasiconvex functions form an interesting class of generalized convex functions, which has numerous practical applications; see, e.g., Penot~\cite{p_2000} and the references therein. Among other things, the survey~\cite{p_2000} discussed dualities related to quasiconvex functions,  subdifferentials of quasiconvex functions, relationships between dualities and subdifferentials, subdifferential calculus, and applications. Concerning continuity of quasiconvex functions, the fundamental result of Crouzeix~\cite{c_1981} says that a real-valued quasiconvex function  on $\mathbb R^n$ is almost everywhere Fr\'echet differentiable, hence almost everywhere continuous. In an infinite-dimensional setting, Rabier~\cite{r_2012} gave necessary and sufficient conditions for a real-valued quasiconvex function $f$ on a Baire topological vector space $X$ (in particular, Banach or Fr\'echet space) to be continuous at the points of a residual subset of~$X$. These conditions involve topological properties of the sublevel sets of~$f$. Recall~\cite[Section~2]{r_2012} that a residual subset of $X$ is the complement of a set of first category.
 
 \medskip
 Note that not only continuity of quasiconvex functions has a significant role but \textit{lower semicontinuity} of quasiconvex functions also has a key role in optimization theory. For instance, a proper lower semicontinuous quasiconvex function  $f$ defined on a reflexive Banach space~$X$ attains a global minimum if it is \textit{coercive} in the sense that $\displaystyle\lim_{\|x\|\to +\infty}f(x)=+\infty$ (see~\cite[Proposition~2.1]{p_2000}). It follows that a proper lower semicontinuous quasiconvex function  $f$ defined on a reflexive Banach space~$X$ attains a minimum on any bounded closed convex set having a nonempty intersection with its effective domain (see~\cite[Corollary of Theorem~1]{p_1966}).
 
  \medskip
  The notion of \textit{strongly quasiconvex function} (see~\cite{glm2025} and Sect.~\ref{Sect_2} below) is a special case of the concept of \textit{uniformly quasiconvex function} given by Polyak~\cite[p.~73]{p_1966}, if one chooses the function $\delta(\tau)=\tau^2$ for all $\tau\geq 0$ as the modulus of the uniform quasiconvexity. First-order characterizations of the strong quasiconvexity of differentiable functions on convex sets (see \cite[Theorem~1]{j_1993} and note that the condition $x,y\in C$ is missing there) follow from the results of Vladimirov et al.~\cite{VNCh_1978}. For twice differentiable functions, the corresponding second-order characterizations were given by Polyak~\cite[p.~73]{p_1966} (see also~\cite[Theorem~2]{j_1993}). The boundedness of the sublevel sets of strongly quasiconvex functions, the strong quasiconvexity of linear-quadratic functions, and the reduction of the strong quasiconvexity of multivariable functions to that of univariable functions with an application to the Euclidean norm were studied by Jovanovi\'c respectively in~\cite{j_1989},~\cite{j_1993}, and~\cite{j_1996}. 
  
  \medskip
  Recently, Lara~\cite[Theorem~1]{Lara_2022} has showed that any extended-real-valued strongly quasiconvex function $f$ defined on $\mathbb R^n$ is \textit{2-supercoercive}, that is $\displaystyle\liminf_{\|x\|\to +\infty}\frac{f(x)}{\|x\|^2}>0$ (in particular, $f$ is coercive). This result improves Theorem~3 in~\cite{j_1989} and leads to the solution existence of  optimization problems defined by lower semicontinuous strongly quasiconvex functions in~\cite[Corollary~3]{Lara_2022}. The lower semicontinuity of strongly quasiconvex functions is also a crucial assumption of Lara's convergence theorem~\cite[Theorem~10]{Lara_2022} for the proximal point algorithm applied to strongly quasiconvex optimization problems.
  
  \medskip
  A series of new results on strong quasiconvexity of functions in infinite dimensions, including a remarkable theorem on optimal value functions, have been obtained by Nam and Sharkansky~\cite{ns_2024}. 
  
  \medskip
  Very recently, a comprehensive survey on strongly quasiconvex functions and strongly quasiconvex optimization problems has been given by Grad et al.~\cite{glm2025}. Among the five potential directions for subsequent research on strongly quasiconvex functions and the related optimization problems formulated by the authors, we are particularly interested in the following one:
  ``1. \textit{Theoretical investigations.} Theoretical exploration of strongly quasiconvex functions remains a promising area of study. For instance, a fundamental open question is whether all real-valued strongly quasiconvex functions defined on $\mathbb R^n$ are necessarily continuous, akin to their convex counterparts. This question arises because all known examples of such functions (see Sect. 3.1) are continuous. Furthermore, identifying new strongly quasiconvex functions with practical relevance could significantly
  enhance the understanding and applicability of this class of functions.'' (see~\cite[p.~36]{glm2025}). For univariable functions, the above fundamental open question has been solved in the negative by Hang et al.~\cite{hly2024} (see Sect.~\ref{Sect_3} below), as well as Nam and Sharkansky~\cite[Examples~6.1 and~6.2]{ns_2024}.
  
  \medskip
  This paper aims at presenting the results on discontinuous univariable strongly quasiconvex  functions from the manuscript~\cite{hly2024} and proposing several types of discontinuous multivariable strongly quasiconvex functions. One of the obtained result says that \textit{for any $n\geq 2$, there exits a strongly quasiconvex function $f:\mathbb R^n\to\mathbb R$, which is discontinuous at infinitely many points}. Since the lower semicontinuity of strongly quasiconvex functions plays an important role in many issues, we will try to understand to which extent a multivariable strongly quasiconvex function can fail to be lower semicontinuous on the interior of its effective domain. The failure of the upper semicontinuity of some kinds of multivariable strongly quasiconvex functions is addressed as well.
  
  \medskip
  The organization of the paper is as follows. In Sect.~\ref{Sect_2}, we recall some definitions and examples. Section~\ref{Sect_3} is devoted to discontinuous univariable strongly quasiconvex functions. Constructions of discontinuous multivariable strongly quasiconvex functions are described in Sect.~\ref{Sect_4}. Concluding remarks are given in Sect.~\ref{Sect_5}. 
  
  \medskip
  Although all the constructions and proofs of Sect.~\ref{Sect_4} are valid for functions defined on Hilbert spaces, herein we prefer the traditional finite-dimensional setting used in the majority of literature on quasiconvex analysis and quasiconvex optimization to the infinite-dimensional setting. 
 
  \section{Preliminaries}\label{Sect_2}
 
 Let $f:\mathbb R^n\to\overline{\mathbb R}$, where $\overline{\mathbb R}=\mathbb R\cup \{\pm \infty\}$, be an extended-real-valued function defined on the Euclidean space $\mathbb R^n$. The scalar product and the norm of the latter space are denoted respectively by $\langle .,.\rangle$ and $\|.\|$.  We assume that $f$ is \textit{proper}, i.e., $f(x)>-\infty$ for all $x\in\mathbb R^n$ and there is some $x\in\mathbb R^n$ with $f(x)\in\mathbb R$. The \textit{effective domain} of $f$ is the set $${\rm dom}\, f:=\{x\in\mathbb R^n\mid f(x)<+\infty\}$$ (see~\cite[p.~5]{Rock_Wets_2009}). By definition, $f$ is said to be \textit{strongly quasiconvex} if there exists $\gamma>0$ such that for any $x,y\in {\rm dom}\, f$ and $t\in (0,1)$ one has
 \begin{equation}\label{sqc}
 	f((1-t)x+ty)\leq \max \{f(x),f(y)\}-\frac{1}{2}\gamma (1-t)t\|x-y\|^2.
 \end{equation}  
 
 From the definition and the assumptions made it follows that if $f$ is strongly quasiconvex, then ${\rm dom}\, f $ is a nonempty convex set. Indeed, for any $x,y\in {\rm dom}\, f$ and $t\in (0,1)$, the inequality~\eqref{sqc} implies that $f((1-t)x+ty)<+\infty$. As $f(x)>-\infty$ for all $x\in\mathbb R^n$, one must have $f((1-t)x+ty)\in\mathbb R$. Thus, the entire line segment $[x,y]$ is contained in ${\rm dom}\, f$.
 
 \medskip
 Given a nonempty set $\Omega\subset\mathbb R^n$ and a function $f:\Omega\to \mathbb R$, one can extend $f$ to a function $\tilde f$ defined on $\mathbb R^n$ by setting $\widetilde f(x)=f(x)$ for $x\in\Omega$ and  $\widetilde f(x)=+\infty$ for $x\in \mathbb R^n\setminus\Omega$. Clearly, $\tilde f$ is proper and ${\rm dom}\, \tilde f=\Omega$. In what follows, we have in mind this canonical extension when dealing with a function defined on a nonempty subset of $\mathbb R^n$. As usual, we say that $f$ is \textit{lower semicontinuous} (l.s.c.) at $\bar x\in\Omega$ if for every $\varepsilon>0$ there is an open set $U\subset\mathbb R^n$ such that $f(x)\geq f(\bar x)-\varepsilon$ for all $x\in\Omega\cap U$. If $\bar x\in\Omega$ and for every $\varepsilon>0$ there is an open set $U\subset\mathbb R^n$ such that $f(x)\leq f(\bar x)+\varepsilon$ for all $x\in\Omega\cap U$, then we say that $f$ is \textit{upper semicontinuous} (u.s.c.) at $\bar x$. If $f$ is both l.s.c. and u.s.c. at $\bar x\in\Omega$, then $f$ is said to be \textit{continuous} at $\bar x$. If $f$ is l.s.c. (resp., u.s.c., or continuous) at every point of $\Omega$, then $f$ is called l.s.c. (resp., u.s.c., or continuous) on $\Omega$.
 
 \medskip
 If~\eqref{sqc} holds for all $x,y$ from a convex set $C\subset {\rm dom}\, f$ and for all $t\in (0,1)$, then we say that $f$ is \textit{strongly quasiconvex on $C$ with the modulus $\gamma$}. In that case, $f$ can have at most one minimizer on $C$. This fact is well known (see, e.g.,~\cite[Corollary~3]{Lara_2022}) and can be verified by using the inequality~\eqref{sqc} and arguing by contradiction. Note that the lower semicontinuity of $f$ is not required in the argument.
 
 \medskip
 Recall that a proper real-valued function $f:\mathbb R^n\to\mathbb R$ is said to be \textit{strongly convex} on a convex set $C\subset {\rm dom}\, f$  if there exists $\gamma>0$ such that \begin{equation}\label{sc}
 	f((1-t)x+ty)\leq (1-t)f(x)+tf(y)-\frac{1}{2}\gamma (1-t)t\|x-y\|^2
 \end{equation} for any $x,y\in {\rm dom}\, f$ and $t\in (0,1)$. 
 
 \medskip
  Since $(1-t)f(x)+tf(y)\leq \max \{f(x),f(y)\}$ for all $x,y\in C\subset {\rm dom}\, f$ and $t\in (0,1)$, the inequality~\eqref{sc} implies the one in~\eqref{sqc}. Hence, strong convexity of a function on a convex set implies its strong quasiconvexity on that set. The reverse implication is invalid.
   
 \begin{example}\label{Ex1}
 	{\rm For $f(x):=\|x\|^2=\langle x,x\rangle$ for every $x\in\mathbb R^n$, where $n\geq 1$, it is easily verified that \begin{equation*}
 			f((1-t)x+ty)=(1-t)f(x)+tf(y)-(1-t)t\|x-y\|^2
 		\end{equation*} for all $x,y\in\mathbb R^n$ and $t\in (0,1)$. So, $f$ is strongly convex on $\mathbb R^n$ with the modulus $\gamma:=2$.}
 \end{example}
 
  \begin{example}\label{Ex1a}
 	{\rm (See~\cite[Remark~18]{Lara_2022}) For every positive natural number $k$, the function $f(x):=\max\big\{\|x\|^{1/2}, \|x\|^2-k\big\}$ is strongly quasiconvex on $\mathbb R^n$  without being convex.}
 \end{example}

 In the sequel, the interior, the closure, and the boundary of a set $\Omega\subset\mathbb R^n$ are denoted respectively by ${\rm int}\,\Omega$, $\overline{\Omega}$, and ${\rm bd}\,\Omega$. Recall that ${\rm bd}\,\Omega=\overline{\Omega}\setminus ({\rm int}\,\Omega).$ The closed unit ball in $\mathbb R^n$ is denoted by $\bar B_n(0,1)$. The closed ball (resp., the open ball) centered at $\bar x\in\mathbb R^n$ with radius $\rho>0$ is denoted by $\bar B(\bar x,\rho)$ (resp., $B(\bar x,\rho)$).
 
 \section{Univariable SQC Functions}\label{Sect_3}
 
 It is well known that a convex function can have discontinuity at some boundary points of its effective domain. Sine a strongly quasiconvex function needs not to be convex and vice versa, discontinuity of a SQC function $f$ on the boundary points of ${\rm dom}\, f$ deserves a careful investigation. The case of univariable SQC functions is considered in the following proposition.
 
 \begin{proposition}\label{Prop1} For any $a,b\in\mathbb R$ with $a<b$, there exists a strongly quasiconvex function $f:[a,b]\to\mathbb R$ which is continuous on the interval $(a,b)$, but discontinuous at both end-points $a$ and $b$ of the interval.
 \end{proposition}
\begin{proof} To prove the proposition, let $f_0:[a,b]\to\mathbb R$ be a continuous, strongly convex function (for example, $f_0(x)=x^2$). Since $f_0$ is strongly convex, there is $\gamma>0$ such that
\begin{equation}\label{str_convex}
	f_0((1-t)x+ty)\leq (1-t)f_0(x)+tf_0(y)-\frac{1}{2}\gamma (1-t)t(x-y)^2
\end{equation} for any $x,y\in [a,b]$ and $t\in (0,1)$.  Put $M=\max\{f_0(x)\mid x\in [a,b]\}$ and define
\begin{equation}\label{function1}
	f(x)=\begin{cases}
		f_0(x) & {\rm for}\ x\in (a,b)\\
		M+1 & {\rm for}\ x\in\{a,b\}.
	\end{cases}	
\end{equation} Take any $x,y\in [a,b]$ with $x<y$ and $t\in (0,1)$. To verify the inequality~\eqref{sqc}, consider the situations: (i) $x,y\in (a,b)$; (ii) $x=a$ and $y\in (a,b)$; (iii) $x\in (a,b)$ and $y=b$; (iv) $x=a$ and $y=b$. If~(i) occurs, then~\eqref{str_convex} implies~\eqref{sqc} because by~\eqref{function1} we have $f_0((1-t)x+ty)=f((1-t)x+ty)$ and
$$(1-t)f_0(x)+tf_0(y)=(1-t)f(x)+tf(y)\leq\max \{f(x),f(y)\}.$$ In situation~(ii), by~\eqref{function1} and~\eqref{str_convex} one has 
\begin{eqnarray*}\begin{array}{rcl}
		f((1-t)a+ty)& = & f_0((1-t)a+ty)\\ &\leq & (1-t)f_0(a)+tf_0(y)-\dfrac{1}{2}\gamma (1-t)t(a-y)^2\\
		& \leq & (1-t)M+tf(y)-\dfrac{1}{2}\gamma (1-t)t(a-y)^2\\
		& \leq & \max \{f(a),f(y)\}-\dfrac{1}{2}\gamma (1-t)t(a-y)^2,
	\end{array}
\end{eqnarray*} which justifies~\eqref{sqc}. Situation~(iii) can be treated similarly. If~(iv) happens, then it follows from~\eqref{function1} and~\eqref{str_convex} that 
\begin{eqnarray*}\begin{array}{rcl}
		f((1-t)a+tb)& = & f_0((1-t)a+tb)\\ &\leq & (1-t)f_0(a)+tf_0(b)-\dfrac{1}{2}\gamma (1-t)t(a-b)^2\\
		& \leq & M-\dfrac{1}{2}\gamma (1-t)t(a-b)^2\\
		& \leq & \max \{f(a),f(b)\}-\dfrac{1}{2}\gamma (1-t)t(a-b)^2.
	\end{array}
\end{eqnarray*} Hence~\eqref{sqc} is valid.
 Since $$\lim\limits_{x\to a}f(x)=f_0(a)< M+1=f(a)$$ and $\lim\limits_{x\to b}f(x)=f_0(b)< M+1=f(b)$, the function $f$ is discontinuous at both end-points $a$ and $b$ of the interval $(a,b)$.
 $\hfill\Box$
 \end{proof}
 
 \begin{remark}{\rm The above construction of a discontinuous SQC function applies when the initial continuous function $f_0:[a,b]\to\mathbb R$ is  strongly quasiconvex. Indeed, let $\gamma>0$ be such that
 \begin{equation}\label{sqc_f0}
 	f_0((1-t)x+ty)\leq \max \{f_0(x),f_0(y)\}-\frac{1}{2}\gamma (1-t)t(x-y)^2
 \end{equation} for any $x,y\in [a,b]$ and $t\in (0,1)$. Define $f$ by~\eqref{function1}, where  $$M:=\max\{f_0(x)\mid x\in [a,b]\},$$ and suppose that $x,y\in [a,b]$, $x<y$, and $t\in (0,1)$ are given arbitrarily. If $x,y\in (a,b)$, then~\eqref{sqc} follows  from~\eqref{sqc_f0}, because $f((1-t)x+ty)=f_0((1-t)x+ty)$, $f(x)=f_0(x)$, and $f(y)=f_0(y)$. If $x=a$ and $y\in (a,b)$, then by~\eqref{function1} and~\eqref{sqc_f0} we get 
\begin{eqnarray*}\begin{array}{rcl}
	f((1-t)a+ty)& = & f_0((1-t)a+ty)\\ &\leq & \max \{f_0(x),f_0(y)\}-\dfrac{1}{2}\gamma (1-t)t(a-y)^2\\
	& \leq & M-\dfrac{1}{2}\gamma (1-t)t(a-y)^2\\
	& \leq & \max \{f(a),f(y)\}-\dfrac{1}{2}\gamma (1-t)t(a-y)^2.
\end{array}
\end{eqnarray*} Similarly, we can show that~\eqref{sqc} holds for $x\in (a,b)$ and $y=b$. If $x=a$ and $y=b$, then~\eqref{function1} and~\eqref{sqc_f0} imply that 
\begin{eqnarray*}\begin{array}{rcl}
	f((1-t)a+tb)& = & f_0((1-t)a+tb)\\ &\leq & \max \{f_0(x),f_0(y)\}-\dfrac{1}{2}\gamma (1-t)t(a-b)^2\\
	& \leq & M-\dfrac{1}{2}\gamma (1-t)t(a-b)^2\\
	& \leq & \max \{f(a),f(b)\}-\dfrac{1}{2}\gamma (1-t)t(a-b)^2.
\end{array}
\end{eqnarray*} Thus, $f$ is a SQC function, which is continuous on $(a,b)$, but discontinuous at both end-points of the interval.
}
\end{remark}

 \begin{proposition}\label{Prop2} There exits a univariable strongly quasiconvex function which is discontinuous at a point in the interior of its effective domain.
 \end{proposition}
\begin{proof} 
Let $f_0:[0,1]\to\mathbb R$ be any continuous, strongly quasiconvex, increasing function (for example, $f_0(x)=x^2$). Choose $\gamma>0$ such that~\eqref{sqc_f0} holds for all $x,y\in [a,b]$ and $t\in (0,1)$. Take any $c\in (0,1)$ and put
\begin{equation}\label{function2}
	f(x)=\begin{cases}
		f_0(x)-1 & {\rm for}\ x\in [0,c]\\
		f_0(x) & {\rm for}\ x\in (c,1].
	\end{cases}	
\end{equation} Clearly, $f$ is lower semicontinuous on $[0,1]$, but discontinuous at $c$. Given any $x,y\in [0,1]$ and $t\in (0,1)$, we distinguish three cases: (i) $x,y\in [0,c]$; (ii) $x,y\in (c,1]$; (iii) $x\in [0,c]$ and $y\in (c,1]$. In the first two cases, the fact that~\eqref{sqc} follows from~\eqref{sqc_f0} and~\eqref{function2} is obvious. Consider the third case. If $(1-t)x+ty\in [0,c]$, then by~\eqref{sqc_f0},~\eqref{function2}, and the increasing property of $f_0$ we see that
\begin{eqnarray*}\begin{array}{rcl}
	f((1-t)x+ty)& = & f_0((1-t)x+ty)-1\\ &\leq & \max \{f_0(x),f_0(y)\}-\dfrac{1}{2}\gamma (1-t)t(x-y)^2-1\\
	& = & f_0(y)-\dfrac{1}{2}\gamma (1-t)t(x-y)^2-1\\
	& =  & f(y)-\dfrac{1}{2}\gamma (1-t)t(x-y)^2-1\\
	& \leq & \max \{f(x),f(y)\}-\dfrac{1}{2}\gamma (1-t)t(x-y)^2;
\end{array}
\end{eqnarray*} thus~\eqref{sqc} is valid. If $(1-t)x+ty\in (c,1]$, then we have \begin{eqnarray*}\begin{array}{rcl}
	f((1-t)x+ty)& = & f_0((1-t)x+ty)\\ &\leq & \max \{f_0(x),f_0(y)\}-\dfrac{1}{2}\gamma (1-t)t(x-y)^2\\
	& = & f_0(y)-\dfrac{1}{2}\gamma (1-t)t(x-y)^2\\
	& = & f(y)-\dfrac{1}{2}\gamma (1-t)t(x-y)^2\\
	& = & \max \{f(x),f(y)\}-\dfrac{1}{2}\gamma (1-t)t(x-y)^2,
\end{array}
\end{eqnarray*} which justifies~\eqref{sqc}.

We have shown that $f$ is a SQC function with the modulus $\gamma$. $\hfill\Box$
\end{proof}

\begin{remark}{\rm The just described construction yields a  strongly quasiconvex function, which is upper semicontinuous on $[0,1]$, but discontinuous at $c$, if one puts 
\begin{equation*}\label{function3}
f(x)=\begin{cases}
	f_0(x)-1 & {\rm for}\ x\in [0,c)\\
	f_0(x) & {\rm for}\ x\in [c,1].
\end{cases}	
\end{equation*}
}
\end{remark}

\begin{proposition}\label{Prop3} There exits a univariable strongly quasiconvex function which is discontinuous at infinitely many points in the interior of its effective domain.
 \end{proposition}
\begin{proof} 
To prove the proposition, it suffices to take any continuous, strongly quasiconvex, increasing function $f_0:[0,1]\to\mathbb R$, which satisfies~\eqref{sqc_f0} for a constant $\gamma>0$, and define
\begin{equation*}\label{function4}
f(x)=\begin{cases}
	f_0(x)-1 & {\rm for}\ x=0\\
	f_0(x)-1+2^{-k+1} & {\rm for}\ x\in \Big((k+1)^{-1},k^{-1}\Big]\ \; {\rm for}\ \; k=1,2,....
\end{cases}	
\end{equation*} Clearly, $f$ is l.s.c. on $[0,1]$, but discontinuous at the points $k^{-1}$ with $k=1,2,....$. The verification of~\eqref{sqc} can be done similarly as the one given in the proof of Proposition~\ref{Prop2}.
 $\hfill\Box$
\end{proof}

\section{Multivariable SQC Functions}\label{Sect_4}

Since the constructions given in Section~\ref{Sect_3} can be used just for univariable functions, discontinuity of a multivariable SQC function on the boundary and interior of its effective domain requires further clarification.

\medskip
The following theorem reveals a possible discontinuity of multivariable SQC functions on the boundaries of their effective domains. 

\begin{theorem}\label{Thm1} For any $n\geq 2$, there exits a strongly quasiconvex function $f:C\to\mathbb R$ with $C\subset\mathbb R$ being a closed and convex set with nonempty interior, which is continuous on ${\rm int}\, C$, but discontinuous at infinitely many points in ${\rm bd}\,C$.
\end{theorem}
\begin{proof} Choose $C=\bar B_n(0,1)$, where $n\geq 2$. Let $f_0(x)=\|x\|^2$ for all $x\in\mathbb R^n$. As it has been noted in Example~\ref{Ex1}, $f_0$ is strongly convex on $\mathbb R^n$ with the modulus $\gamma=2$. Moreover, one has 
\begin{equation}\label{sc_f0}
		f_0((1-t)x+ty)=(1-t)f_0(x)+tf_0(y)-(1-t)t\|x-y\|^2
\end{equation} for all $x,y\in\mathbb R^n$ and $t\in (0,1)$. 

Let $A$ be such a subset of ${\rm bd}\,C$, that both $A$ and ${\rm bd}\,C\setminus A$ are dense in the \textit{relative topology} (see~\cite[p.~51]{Kelley_1955}) of ${\rm bd}\,C\subset\mathbb R^n$ (for instance, one can take $$A=\big\{x=(x_1,\dots,x_n)\in {\rm bd}\,C\mid x_1\in\mathbb Q\big\}).$$ Define 
\begin{equation}\label{function5}
	f(x)=\begin{cases}
		f_0(x) & {\rm for}\ x\in C\setminus A\\
		f_0(x)+\alpha & {\rm for}\ x\in A,
	\end{cases}	
\end{equation} where $\alpha>0$ is an arbitrary constant. Clearly, this function $f:C\to\mathbb R$ is continuous on ${\rm int}\, C$, u.s.c. on $C$, but discontinuous at every point in ${\rm bd}\,C$. Note in addition that $f_0(x)\leq f(x)$ for all $x\in C$. 

To show that $f$ is strongly quasiconvex on $C$ with the modulus $\gamma=2$, we take any distinct points $x,y\in C$ and any value $t\in (0,1)$. If both points $x,y$ belong to ${\rm int}\, C$, then the inequality~\eqref{sqc} with $\gamma:=2$ is the consequence of~\eqref{function5} and the above-mentioned strong convexity of $f_0$. Now, suppose that $x\in {\rm int}\, C$ and $y\in {\rm bd}\,C$. Since $f((1-t)x+ty)=f_0((1-t)x+ty)$ by~\eqref{function5}, from~\eqref{sc_f0} we can deduce that
\begin{eqnarray}\label{estimates_1}
\begin{array}{rcl}
	f((1-t)x+ty) & = & (1-t)f_0(x)+tf_0(y)-(1-t)t\|x-y\|^2\\
	& \leq & \max\big\{f_0(x),f_0(y)\big\}-(1-t)t\|x-y\|^2\\
	& \leq & \max\big\{f(x),f(y)\big\}-(1-t)t\|x-y\|^2.
\end{array}
\end{eqnarray} Thus,~\eqref{sqc} holds with $\gamma=2$. The analysis of the situation where  $y\in {\rm int}\, C$ and $x\in {\rm bd}\,C$ is similar. Next, suppose that both points $x$ and $y$ belong to ${\rm bd}\,C$. As $$(1-t)x+ty\in {\rm int}\, C\subset C\setminus A,$$ one has $f((1-t)x+ty)=f_0((1-t)x+ty)$. So,~\eqref{sc_f0} assures that the estimates in~\eqref{estimates_1} are valid. Therefore, again we get~\eqref{sqc} with $\gamma=2$. 

The proof is complete.
$\hfill\Box$
\end{proof}

In connection with Proposition~\ref{Prop2}, a natural question arises: \textit{Whether there exits a mutivariable strongly quasiconvex function which is discontinuous at a point in the interior of its effective domain, or not?} 

\medskip
To solve the above question, one cannot apply the construction of an increasing univariable strongly quasiconvex function used for proving Proposition~\ref{Prop2}. The next theorem deals with the question by using another construction.  

\begin{theorem}\label{Thm2} For any $n\geq 2$, there exits a strongly quasiconvex function $f:\mathbb R^n\to\mathbb R$, which is discontinuous at a point $c\in \mathbb R^n$, but continuous on the set $\mathbb R^n\setminus\{c\}$.
\end{theorem}
\begin{proof} Let $f_0(x)=\|x\|^2$ for all $x\in\mathbb R^n$. As it has been noted before, from~\eqref{sc_f0} we have
\begin{equation}\label{sc_impl} f_0((1-t)x+ty)\leq \max\big\{f_0(x),f_0(y)\big\})-(1-t)t\|x-y\|^2
			\end{equation}
for all $x,y\in\mathbb R^n$ and $t\in (0,1)$.
	
	Put 
	\begin{equation}\label{function6}
		f(x)=\begin{cases}
			f_0(x) & {\rm for}\ x\neq 0\\
			-\alpha & {\rm for}\ x=0,
		\end{cases}	
	\end{equation} where $\alpha>0$ can be chosen arbitrarily. Clearly, $f$ is continuous on the set $\mathbb R^n\setminus\{0\}$, l.s.c. at $0$, but not u.s.c. at $0$. 
	
	To show that $f$ is strongly quasiconvex on $\mathbb R^n$, we pick any distinct points $x,y\in \mathbb R^n$ and a value $t\in (0,1)$. If $0\notin \{x,\, y,\, (1-t)x+ty\}$, then by~\eqref {function6} we have $$f(x)=f_0(x),\ f(y)=f_0(y),\ f((1-t)x+ty)=f_0((1-t)x+ty).$$ Hence,~\eqref{sc_impl} implies~\eqref{sqc} with $\gamma:=2$. If $x=0$, then combining~\eqref{function6} with~\eqref{sc_impl} yields
	\begin{eqnarray*}\label{estimates_2}
	\begin{array}{rcl}
		f((1-t)x+ty) & = & f_0((1-t)x+ty)\\
		& \leq & \max\big\{f_0(x),f_0(y)\big\}-(1-t)t\|x-y\|^2\\
		& = & \max\big\{0,f(y)\big\}-(1-t)t\|x-y\|^2\\
		& = & f(y)-(1-t)t\|x-y\|^2\\
		& = & \max\big\{f(x),f(y)\big\}-(1-t)t\|x-y\|^2.		
	\end{array}
	\end{eqnarray*} Therefore,~\eqref{sqc} holds with $\gamma=2$. The case $y=0$ can be considered in a similar way. If $(1-t)x+ty=0$, then $$f((1-t)x+ty)=-\alpha<f_0(0)=f_0((1-t)x+ty).$$ Therefore, by~\eqref{sc_impl} and ~\eqref{function6} we can infer that~\eqref{sqc} holds with $\gamma=2$.
	
	We have thus proved that $f$ is strongly quasiconvex on $\mathbb R^n$ with the modulus $\gamma=2$.
$\hfill\Box$
\end{proof}

\begin{remark}{\rm The construction of the discontinuous SQC function and arguments in the proof of Theorem~\ref{Thm2} apply well to the situation where $f_0:C\to \mathbb R$ is a continuous function defined on a convex set $C\subset\mathbb R^n$, provided that ${\rm int}\, C\neq\emptyset$ and $f$ attains its global minimum on $C$ at a point $c\in {\rm int}\, C$. To be more precise, instead of~\eqref{function6}, the desired function $f:C\to \mathbb R$ is given by \begin{equation*}\label{function7}
			f(x)=\begin{cases}
				f_0(x) & {\rm for}\ x\in C\setminus \{c\}\\
				f_0(c)-\alpha & {\rm for}\ x=c,
			\end{cases}	
		\end{equation*} where $\alpha>0$ is an arbitrarily chosen constant. This function $f$ is strongly quasiconvex on $C$, continuous on the set $C\setminus\{c\}$, l.s.c. at $c$, but not u.s.c. at $c$. 
	}
\end{remark}

Looking back to Proposition~\ref{Prop3}, one may ask: \textit{Is it possible to construct a multivariable strongly quasiconvex function which is discontinuous at infinitely many points in the interior of its effective domain?} 

\medskip
Using an idea quite different from that one for proving Theorem~\ref{Thm2}, we will solve the above question in the affirmative. 

\begin{theorem}\label{Thm3} For any $n\geq 2$, there exits a strongly quasiconvex function $f:\mathbb R^n\to\mathbb R$, which is discontinuous at infinitely many points.
\end{theorem}
\begin{proof} Suppose that $n\geq 2$. As in the proof of Theorem~\ref{Thm2}, we 
consider the function $f_0(x)=\|x\|^2$, $x\in\mathbb R^n$. Let $\rho$ and $\beta$ be two arbitrarily chosen positive constants. Define
\begin{equation}\label{function_n2}
	g(x)=\begin{cases}
		f_0(x) & {\rm if}\ \|x\|\leq \rho\\
		f_0(x)+\beta & {\rm if}\ \|x\|>\rho.
	\end{cases}	
\end{equation} It is easy to see that the just defined function $g:\mathbb R^n\to \mathbb R$ is continuous on the set $\big\{z\in \mathbb R^n\mid \|z\|\neq\rho\big\}$ and l.s.c. at every point belonging to \begin{equation}\label{D} D:=\big\{z\in \mathbb R^n\mid \|z\|=\rho\big\}.\end{equation} It is also clear that $g$ is not u.s.c. at every point in $D$. Thus, in particular, $g$ is discontinuous at every point in $D$.

Our function $g$ is SQC on $\mathbb R^n$. Indeed, for any distinct points $x,y\in \mathbb R^n$ and any value $t\in (0,1)$, by swapping the roles of $x$ and $y$ (if necessary) we see that there are the following four cases:

 (i) $\|x\|\leq \rho$ and $\|y\|\leq \rho$;
  
 (ii)  $\|x\|>\rho$ and $\|y\|>\rho$;
 
 (iii) $\|x\|=\rho$  and $\|y\|>\rho$;
 
 (iv) $\|x\|<\rho$  and $\|y\|>\rho$.
 
 Let $h(x)=f_0(x)+\beta=\|x\|^2+\beta$ for $x\in\mathbb R^n$. Since the function $h:\mathbb R^n\to\mathbb R$ is strongly convex on $\mathbb R^n$ with the modulus $\gamma=2$, one has 
 \begin{equation}\label{sc_impl_h} h((1-t)x+ty)\leq \max\big\{h(x),h(y)\big\})-(1-t)t\|x-y\|^2
 \end{equation}
 for all $x,y\in\mathbb R^n$ and $t\in (0,1)$. To shorten some subsequent writings, put $$z_t=(1-t)x+ty.$$
 
 If the case~(i) occurs, then from~\eqref{sc_impl} and~\eqref{function_n2} we get
 \begin{equation}\label{sqc_g} g(z_t)\leq \max\big\{g(x),g(y)\big\}-(1-t)t\|x-y\|^2.
 \end{equation}
 
 In the case~(ii), if $\|z_t\|>\rho$, then combining~\eqref{function_n2} and~\eqref{sc_impl_h} yields~\eqref{sqc_g}. If $\|z_t\|\leq\rho$, then from~\eqref{sc_impl} and~\eqref{function_n2} we have
 \begin{eqnarray*}\label{estimates_3}
 	\begin{array}{rcl}
 		g(z_t) & = & f_0(x_t)\\
 		& \leq & \max\big\{f_0(x),f_0(y)\big\}-(1-t)t\|x-y\|^2\\
 		& = & \max\big\{h(x),h(y)\big\}-(1-t)t\|x-y\|^2 -\beta\\
 		& = & \max\big\{g(x),g(y)\big\}-(1-t)t\|x-y\|^2 -\beta.		
 	\end{array}
 \end{eqnarray*} This implies~\eqref{sqc_g}.
 
 If the case~(iii) takes place, then either $\|z_t\|>\rho$ or $\|z_t\|\leq\rho$. Assuming that $\|z_t\|>\rho$, from~\eqref{function_n2} and~\eqref{sc_impl_h} we can deduce that
 \begin{eqnarray*}\label{estimates_z_1}
 	\begin{array}{rcl}
 		g(z_t) & = & h(z_t)\\
 		& \leq & \max\big\{h(x),h(y)\big\}-(1-t)t\|x-y\|^2\\
 		& = & h(y)-(1-t)t\|x-y\|^2\\
 		& = & g(y)-(1-t)t\|x-y\|^2\\
 		& = & \max\big\{g(x),g(y)\big\}-(1-t)t\|x-y\|^2.		
 	\end{array}
 \end{eqnarray*} Thus, the inequality~\eqref{sqc_g} holds true. Now, supposing that $\|z_t\|\leq\rho$, combining~\eqref{function_n2} with~\eqref{sc_impl} we get 
 \begin{eqnarray*}\label{estimates_z_1a}
 \begin{array}{rcl}
 	g(z_t) & = & f_0(x_t)\\
 	& \leq & \max\big\{f_0(x),f_0(y)\big\}-(1-t)t\|x-y\|^2\\
  	& < & \max\big\{g(x),g(y)\big\}-(1-t)t\|x-y\|^2.		
 \end{array}
\end{eqnarray*} So, again the relation~\eqref{sqc_g} is valid.
 
 Finally, consider the case~(iv). Since $g(z_t)\leq h(z_t)$ by~\eqref{function_n2}, thanks to~\eqref{sc_impl_h} we have
 \begin{eqnarray*}\label{estimates_z_2}
 	\begin{array}{rcl}
 		g(z_t) & \leq & h(x_t)\\
 		& \leq & \max\big\{h(x),h(y)\big\}-(1-t)t\|x-y\|^2\\
 		& = & h(y)-(1-t)t\|x-y\|^2\\
 		& = & g(y)-(1-t)t\|x-y\|^2\\
 		& = & \max\big\{g(x),g(y)\big\}-(1-t)t\|x-y\|^2,		
 	\end{array}
 \end{eqnarray*} where the last equality is valid because $\|x\|<\rho$  and $\|y\|>\rho$.
 
 Summing up all the above, we can conclude that $g$ is strongly quasiconvex on $\mathbb R^n$ with the modulus $\gamma=2$.
 $\hfill\Box$
\end{proof}

\begin{remark}\label{rem4.2} {\rm The strongly quasiconvex function $g$ obtained in the proof of Theorem~\ref{Thm3} is l.s.c. on $\mathbb R^n$. Slightly modifying the formula of $g$, one can get a strongly quasiconvex function which is u.s.c. on the whole space~$\mathbb R^n$, but it is \textit{not l.s.c. at every point from the sphere $D$} given by~\eqref{D}. To do so, it suffices to put
		\begin{equation}\label{function_n3}
		g(x)=\begin{cases}
			f_0(x) & {\rm if}\ \|x\|<\rho\\
			f_0(x)+\beta & {\rm if}\ \|x\|\geq\rho.
		\end{cases}	
		\end{equation} To show that this function is strongly quasiconvex on $\mathbb R^n$ with the modulus $\gamma=2$, one can apply the arguments similar to those of the preceding proof. To be more precise, given any distinct points $x,y\in \mathbb R^n$ and a value $t\in (0,1)$, by swapping the roles of $x$ and $y$ (if necessary) we have the following four cases:
		
		(i') $\|x\|<\rho$ and $\|y\|<\rho$;
		
		(ii')  $\|x\|\geq\rho$ and $\|y\|\geq\rho$;
		
		(iii') $\|x\|<\rho$  and $\|y\|=\rho$;
		
		(iv') $\|x\|<\rho$  and $\|y\|>\rho$.
		
		Using~\eqref{function_n3} and the auxiliary function~$h$ introduced in the proof of Theorem~\ref{Thm3}, we can analyze the cases (i')--(iv') by the arguments quite similar to those of that proof.}
\end{remark}

Combining the constructions of the discontinuous multivariable SQC functions described by the formulas~\eqref{function_n2} and~\eqref{function_n3}, we can obtain the following result.

\begin{theorem}\label{Thm4} For any $n\geq 2$, there exits a strongly quasiconvex function $f:\mathbb R^n\to\mathbb R$, which is lower semicontinuous but not upper semicontinuous at infinitely many points, as well as upper semicontinuous but not lower semicontinuous at infinitely many points.
\end{theorem}
\begin{proof} Fix any $n\geq 2$. Let $f_0$, $\rho>0$, $\beta>0$, $D$, and $h$ be the same as in the proof of Theorem~\ref{Thm3}. Put $\Omega_0=B(0,\rho)$, $\Omega_1=\mathbb R^n\setminus \bar B(0,\rho)$, $$D_0=\{x=(x_1,\dots,x_n)\in D\mid x_1\in\mathbb Q\},$$ and  $D_1=D\setminus D_0$. Note that $D_0$ and $D_1$ are dense subsets of $D$ in the relative topology of $D\subset\mathbb R^n$. Consider the function 
	\begin{equation}\label{function_n4}
		g(x)=\begin{cases}
			f_0(x) & {\rm for}\ x\in\Omega_0\cup D_0\\
			f_0(x)+\beta & {\rm for}\ x\in\Omega_1\cup D_1.
		\end{cases}	
	\end{equation} It is a simple matter to verify this function~$g$ is continuous on the set~$\Omega_0\cup \Omega_1$, l.s.c. but not u.s.c. at every point of~$D_0$, while u.s.c. but not l.s.c. at every point of~$D_1$. To complete the proof, it remains to show that~$g$ is is strongly quasiconvex on~$\mathbb R^n$.
	
	Let two distinct points $x,y\in \mathbb R^n$ and a value $t\in (0,1)$ be given arbitrarily. Set $z_t=(1-t)x+ty.$
	
	If $x,y\in\Omega_0\cup D_0$, then $z_t\in\Omega_0$. Hence, in accordance with~\eqref{function_n4}, one has $g(x)=f_0(x)$, $g(y)=f_0(y)$, and $g(z_t)=f_0(z_t)$. So, using~\eqref{sc_impl} gives~\eqref{sqc_g}.
	
	If $x,y\in\Omega_1\cup D_1$, then $\|x\|\geq\rho$ and $\|y\|\geq\rho$. By swapping the roles of $x$ and $y$ (if necessary), we may assume that $\|x\|\leq \|y\|$. Hence, $h(x)\leq h(y)$. If $z_t\in\Omega_0\cup D_0$, then combining~\eqref{function_n4} with~\eqref{sc_impl_h} yields
	\begin{eqnarray}\label{estimates_z_3}
		\begin{array}{rcl}
			g(z_t) & = & f_0(z_t)\\
			& \leq & h(z_t)\\
			& \leq & \max\big\{h(x),h(y)\big\}-(1-t)t\|x-y\|^2\\
			& = & h(y)-(1-t)t\|x-y\|^2\\
			& = & g(y)-(1-t)t\|x-y\|^2\\
			& = & \max\big\{g(x),g(y)\big\}-(1-t)t\|x-y\|^2.		
		\end{array}
	\end{eqnarray} So, the inequality~\eqref{sqc_g} is valid. If  $z_t\in\Omega_1\cup D_1$, then by~\eqref{function_n4} and~\eqref{sc_impl_h} we have
		\begin{eqnarray}\label{estimates_z_4}
		\begin{array}{rcl}
			g(z_t) & = & h(z_t)\\
			& \leq & \max\big\{h(x),h(y)\big\}-(1-t)t\|x-y\|^2\\
			& = & h(y)-(1-t)t\|x-y\|^2\\
			& = & g(y)-(1-t)t\|x-y\|^2\\
			& = & \max\big\{g(x),g(y)\big\}-(1-t)t\|x-y\|^2.		
		\end{array}
	\end{eqnarray} This establishes~\eqref{sqc_g}.
	
	Finally, let us consider the situation one of the points $x$ or $y$ is contained in $\Omega_0\cup D_0$ and another point is contained in $\Omega_1\cup D_1$. By swapping the roles of $x$ and $y$ (if necessary), we may assume that  $x\in\Omega_0\cup D_0$ and $y\in \Omega_1\cup D_1$. Then we have $\|x\|\leq\rho\leq \|y\|$. This implies that $h(x)\leq h(y)$. Therefore, if $z_t\in\Omega_0\cup D_0$ then the estimates~\eqref{estimates_z_3} are valid, and if $z_t\in\Omega_1\cup D_1$ then the estimates~\eqref{estimates_z_4} hold true. So, in any case, the inequality~\eqref{sqc_g} takes place.
	
	Thus, we have shown that the function $g$ given by~\eqref{function_n4} is strongly quasiconvex on $\mathbb R^n$ with the modulus $\gamma=2$.	
$\hfill\Box$
\end{proof}

\begin{remark}{\rm A careful analysis of the last two proofs and of Remark~\ref{rem4.2} tells us that all the arguments used there remain effective if instead of the strongly convex function $f_0(x)=\|x\|^2$ we consider any continuous SQC function $f_0:\mathbb R^n\to\mathbb R$ having the property $f_0(x)\leq f_0(y)$ for all $x,y\in\mathbb R^n$ with $\|x\|\leq\|y\|$ (note that all the functions discussed in Example~\ref{Ex1a} fulfill the required conditions). Of course, the inequality~\eqref{sqc_g} then becomes \begin{equation*}\label{sqc_g_1} g(z_t)\leq \max\big\{g(x),g(y)\big\}-\frac{1}{2}\gamma(1-t)t\|x-y\|^2,
\end{equation*} if the new function $f_0$ is strongly quasiconvex on $\mathbb R^n$ with a modulus $\gamma>0$.
}
\end{remark}

\begin{remark}\label{rem4.4} {\rm
		Since one usually considers SQC functions with bounded effective domains, the following modifications of the formulas for discontinuous multivariable SQC functions used in the proof of Theorem~\ref{Thm3}, Remark~\ref{rem4.2}, and the proof of Theorem~\ref{Thm4}, is of interest. 
		
		Let $n\geq 2$. Take any bounded convex set $C\subset\mathbb R^n$ such that $({\rm int}\,C)\cap D\neq\emptyset$, where $D$ is defined by~\eqref{D}.   
		
		For the function $g$ given by~\eqref{function_n2}, the modified function $g_1:\mathbb R^n\to \overline{\mathbb R}$ is defined by setting $g_1(x)=g(x)$ for all $x\in C$ and $g_1(x)=+\infty$ for all $x\notin C$. From the proof of Theorem~\ref{Thm3} it follows that $g_1$ is strongly quasiconvex with the modulus $\gamma=2$, l.s.c. at every point in $C\cap D$, but not u.s.c. at every point in $({\rm int}\,C)\cap D$. Thus, $g_1$ is discontinuous at infinitely many points in the interior of its effective domain.
		
		Now, consider the function $g$ given by~\eqref{function_n3} and put $g_2(x)=g(x)$ for all $x\in C$ and $g_2(x)=+\infty$ for all $x\notin C$. The analysis given in Remark~\ref{rem4.2} shows that the function $g_2:\mathbb R^n\to \overline{\mathbb R}$ is strongly quasiconvex with the modulus $\gamma=2$, u.s.c. at every point in $({\rm int}\,C)\cap D$, but not l.s.c. at every point in $({\rm int}\,C)\cap D$. So, $g_2$ is discontinuous at  infinitely many points in the interior of its effective domain.
		
		Finally, to modify the function $g$ given by~\eqref{function_n4}, define $g_3(x)=g(x)$ for all $x\in C$ and $g_3(x)=+\infty$ for all $x\notin C$. From the proof of Theorem~\ref{Thm3} it is clear that the function $g_3:\mathbb R^n\to \overline{\mathbb R}$ is strongly quasiconvex with the modulus $\gamma=2$, continuous on the set~$({\rm int}\,C)\cap \big(\Omega_0\cup \Omega_1\big)$, l.s.c. but not u.s.c. at every point of~$({\rm int}\,C)\cap D_0$, while u.s.c. but not l.s.c. at every point of~$({\rm int}\,C)\cap D_1$. Therefore, $g_3$ is discontinuous at  infinitely many points in the interior of its effective domain.
		}
\end{remark}  

The next simple proposition allows us to apply invertible affine transformations to all the discontinuous SQC functions mentioned in this section to have many other discontinuous SQC functions, whose effective domains are of various shapes and located at different places in the Euclidean spaces under consideration.

\begin{proposition}\label{Prop4} Let  $C\subset\mathbb R^n$ be a convex set, $f:C\to\mathbb R$ a strongly quasiconvex function with a modulus $\gamma>0$. Suppose that $A:\mathbb R^n\to\mathbb R^n$ is an invertible linear operator and $b\in\mathbb R^n$. Then, the formula $g(z)=f\big(A^{-1}(z-b)\big)$, where $z\in A(C)+b$, gives a strongly quasiconvex function defined on the convex set $A(C)+b$.	
\end{proposition}
\begin{proof} Consider the constant $\beta:=\min\left\{\|A^{-1}(v)\|\mid v\in\mathbb R^n,\; \|v\|=1\right\}$, which is equal to the \textit{minimum modulus} of the linear operator $A^{-1}:\mathbb R^n\to\mathbb R^n$ (see~\cite[pp.~404--405]{in2020}). Clearly, $\beta>0$ and one has $\|A^{-1}(v)\|\geq \beta\|v\|$ for all $v\in \mathbb R^n$. Hence, for every $v\in \mathbb R^n$, it holds that \begin{equation}\label{ineq_beta}
		-\|A^{-1}(v)\|^2\leq -\beta^2\|v\|^2.
	\end{equation} 
	
	Let $z,w\in A(C)+b$ and $t\in (0,1)$ be given arbitrarily. Set $x=A^{-1}(z-b)$ and $y=A^{-1}(w-b)$. Using the strong quasiconvexity of $f$ and the inequality~\eqref{ineq_beta}, we have  
	\begin{eqnarray*}\begin{array}{rcl}
				g((1-t)z+tw) & = & f\left(A^{-1}((1-t)z+tw-b)\right)\\
				 & = & f\left((1-t)A^{-1}(z-b)+tA^{-1}(w-b)\right) \\
				 & = & f((1-t)x+ty)\\
			  &\leq & \max \{f(x),f(y)\}-\frac{1}{2}\gamma (1-t)t\|x-y\|^2\\
			  & = & \max \{g(z),g(w)\}-\frac{1}{2}\gamma (1-t)t\|A^{-1}(z-b)-A^{-1}(w-b)\|^2\\
			  & \leq & \max \{g(z),g(w)\}-\frac{1}{2}\gamma (1-t)t\|A^{-1}(z-w)\|^2\\ 
			  & \leq & \max \{g(z),g(w)\}-\frac{1}{2}\gamma\beta^2 (1-t)t\|z-w\|^2.
		\end{array}	
	\end{eqnarray*} We have thus shown that the function $g$, which is defined on $A(C)+b$, is strongly quasiconvex with the modulus $\gamma\beta^2$.
	$\hfill\Box$
\end{proof}

\section{Conclusions}\label{Sect_5}

To answer in details the important open question from~\cite[p.~36]{glm2025},  several constructions for building discontinuous strongly quasiconvex functions have been given in this paper.

\medskip
The existence of real-valued strongly quasiconvex functions, which are lower semicontinuous but not upper semicontinuous at infinitely many points, as well as upper semicontinuous but not lower semicontinuous at infinitely many points has been proved.

\medskip
All the obtained discontinuous strongly quasiconvex functions come in full agreement with the classical result of Crouzeix~\cite{c_1981} saying that a real-valued quasiconvex function on $\mathbb R^n$ is almost everywhere Fr\'echet differentiable, hence almost everywhere continuous.

\medskip
The constructions and proofs of Sect.~\ref{Sect_4} can yield discontinuous strongly quasiconvex functions defined on Hilbert spaces.

\section*{Acknowledgements}

The research of Nguyen Thi Van Hang and Nguyen Dong Yen was supported by the project NCXS02.01/24-25 of Vietnam Academy of Science and Technology. The research of Felipe Lara was partially supported by Anid-Chile through Fondecyt Regular 1241040. The authors are indebted to Vietnam Institute for Advanced Study in Mathematics (VIASM) for warm hospitality during their recent research stays there.

\section*{Declarations}

\textbf{Conflict of interest} The authors have not disclosed any Conflict of   interest.

\smallskip
\noindent \textbf{Data availability statement} This manuscript has no associated data.

\end{document}